\documentclass[a4paper]{article}
\usepackage{graphicx} 
\usepackage{float}
\usepackage{amsmath}
\usepackage{amsthm}
\usepackage{url}
\usepackage{tikz}
\usetikzlibrary{calc}
\usepackage[T1]{fontenc}
\usepackage[utf8]{inputenc}

\title{Circular chromatic index of small graphs}
\author{Ján Mazák, Filip Zrubák}
\author{
Ján Mazák, Filip Zrubák
\\[3mm]
\\{\tt jan.mazak@fmph.uniba.sk}
\\[5mm]
Comenius University, Mlynská dolina, 842 48 Bratislava\\
}
\date{}

\theoremstyle{definition}
\newtheorem{theorem}{Theorem}
\newtheorem{lemma}{Lemma}
\newtheorem{problem}{Problem}
\newtheorem{conjecture}{Conjecture}

\usepackage{setspace}
\usepackage{xcolor}

\usepackage{geometry}
\def\cchi{\chi_c'}

\begin{document}

\maketitle

\begin{abstract}
We systematically determine circular chromatic index of small graphs and multigraphs with maximum degree $4$, $5$, $6$ (and also their number for a given small order). We construct several infinite families of such graphs with circular chromatic index in the set $\{\Delta + 1/2, \Delta + 2/3, \Delta + 3/4$, $\Delta + 1\}$. Our results refute edge-connectivity variants of the ``Upper Gap Conjecture'' (about the non-existence of graphs with circular chromatic index just below $\Delta + 1$).
\end{abstract}

\section{Introduction}

Circular chromatic index $\cchi$ has been investigated for more than two decades \cite{hackmann}, yet many questions remain open. For graphs with maximum degree $\Delta\le 2$, the situation is fully understood, and for $\Delta = 3$, all low-hanging fruit has been picked. This paper examines finite graphs and multigraphs with $\Delta\in \{4,5,6\}$.

A circular $r$-colouring $c$ of a graph $G$ is an assignment of a real number from $[0,r)$ (i.e. a colour) to each vertex of $G$ such that any two adjacent vertices $u$, $v$ satisfy $1\le |c(u) - c(v)|\le r-1$. An equivalent viewpoint is that we wrap the interval $[0,r)$ along a \emph{colour circle} of circumference $r$ and require that colours of adjacent vertices are at distance at least one along the circle (in both directions). The infimum of all values $r$ such that a graph $G$ is circularly $r$-colourable is its \emph{circular chromatic number} $\chi_c(G)$; for finite graphs, $\chi_c$ is rational and attained \cite{zhusurvey2001}.

A circular $r$-edge-colouring of a graph is a circular $r$-rolouring of its line graph; the circular chromatic index $\cchi(G)$ of a graph $G$ is the least $r$ such that $G$ admits a circular $r$-edge-colouring. Since circular colourings are a straightforward generalization of proper integer colourings, we will usually drop the word ``circular'' in the subsequent discussions. The notion of circular edge-colouring naturally generalizes to multigraphs.

The primary open question is the structure of the set of possible values of $\cchi$. It is known \cite{zhusurvey2006} that $\chi'(G) = \lfloor \cchi (G)\rfloor$ for any graph $G$. If $G$ is simple, $\cchi(G)\in [\Delta, \Delta + 1]$. The lower end of this interval is completely covered:
\begin{itemize}
    \item There exists a $3$-regular graph with $\cchi=r$ for any rational $r\in [3, 3+1/3]$ \cite{mazaklukotka}.
    \item For any odd $\Delta\ge 5$, there exists a $\Delta$-regular graph $G$ with $\cchi=r$ for any rational $r\in [\Delta, \Delta+1/4]$. If $G$ is allowed to contain parallel edges, the interval can be extended to $[\Delta, \Delta+1/3]$. \cite{Lin2013}
    \item For any even $\Delta\ge 4$, there exists a $\Delta$-regular graph $G$ with $\cchi=r$ for any rational $r\in [\Delta, \Delta+1/6]$. If $G$ is allowed to contain parallel edges, the interval can be extended to $[\Delta, \Delta+1/3]$. \cite{Lin2015}
\end{itemize}
Apart from these general constructions, there are some more specific ones \cite{candrakova2014,mazaklukotka} covering values slightly below or above $\Delta + 1/3$. The interval $[\Delta + 1/3, \Delta + 1/2]$ seems particularly challenging for $\Delta=3$: there are infinitely many values of $\cchi$ in the interval, yet it is likely that ``most'' values are not attained.

In the upper end of the interval near $\Delta + 1$, which is the focus of this paper, it seems there are no graphs.
\begin{itemize}
    \item There is no graph $G$ with $\cchi(G)\in (3 - 1/2, 3)$ \cite{zhusurvey2006}.
    \item There is no graph $G$ with $\cchi(G)\in (4 - 1/3, 4)$ \cite{afshani}.
    \item ``Upper Gap Conjecture'' (e.g. Problem 1 in \cite{Kaiser2004}, Conjecture 4.3 in \cite{zhusurvey2006}): given an integer $k\ge 4$, there are no graphs with $\cchi(G)\in [k+1-1/k, k+1]$.
\end{itemize}
While we are mostly concerned with the Upper Gap Conjecture, most of our discussion also pertains to the middle portion of the interval $[\Delta, \Delta+1]$, especially the parts highlighting why the problem is difficult. 

One would expect the situation to be more complex for any $\Delta > 3$ compared to subcubic graphs because of the following behaviour of circular edge-colourings in the neighbourhood of a vertex. Imagine a graph $G$ with $\Delta = 3$ and a circular $r$-edge-colouring for $r$ just slightly below $4$. Pick any vertex $v$ of $G$ and assume two of its incident edges are already coloured. Their colours split the colour circle of length $r$ into two arcs. The longer of the arcs will very likely have length at least $2$ since $r$ is close to $4$, and that is sufficient to accommodate a colour for the third edge at $v$. The two given colours would need to be almost precisely opposite each other to remove the possibility of colouring the third edge (only one simple graph and one additional multigraph force this situation \cite{afshani}). For larger $\Delta$, however, even at a vertex of degree $d$ much smaller than $\Delta$, it would be challenging to colour the remaining edge after $d-1$ edges were coloured if their colours are spread apart (with circular distances slightly below $2$).
One would thus not expect subcubic graphs with multiple vertices of degree $2$ and an interesting value of $\cchi$, yet for $\Delta > 3$ vertices of small degree are not uncommon (e.g. vertices of degree $3$ appear in graphs with $\Delta = 5$ and $\cchi = 5 + 3/4$). This is one of the reasons why we extended our search beyond regular graphs.

There are large classes of graphs with $\cchi$ dictated by a small subgraph, and thus not of much interest. Investigation should be focused on those that are \emph{$\cchi$-critical}: any proper subgraph $G'$ of such a graph $G$ satisfies $\cchi(G')<\cchi(G)$. For instance, we have a complete characterisation of $\cchi$-critical multigraphs $G$ with $\Delta(G)\le 3$ and $\cchi(G)=4$: there are two of them, and they arise by subdividing an edge in either $K_4$ or a multigraph with $2$ vertices and $3$ parallel edges \cite{afshani}.

Any attempt to prove a variant of the Upper Gap Conjecture (i.e. a statement about graphs being circularly $r$-edge-colourable for some $r < \Delta + 1$) must account for the exceptions having $\cchi=\Delta+1$. These exceptions appear quite numerous for $\Delta>3$ and can be fairly large. We demonstrate some infinite families of them in Section~\ref{sec:maxcchi}, yet graphs in those families are not $\cchi$-critical, so the search continues.

The methods and notions related to circular chromatic index used in this paper are standard; an interested reader can find them e.g. in \cite{zhusurvey2001,zhusurvey2006}. The \emph{circular distance} of two colours $x$, $y$ in a circular $r$-edge-colouring is defined as $d_r(x,y) = \min\{|x-y|, r-|x-y|\}$.

A common efficient way for encoding graphs and multigraphs are the graph6 and sparse6 formats; details can be found in \cite{McKay2014}. We use these formats in this paper to describe specific graphs since they are most useful if a reader wanted to analyze the graphs using computer tools.

\section{Computational methods}

\def\cchiless{CCHI$^\le_r$\ }
\def\cchiequal{CCHI$^=_r$\ }

We start by discussing theoretical complexity of determining circular chromatic index. Let $r\ge 3$ be a rational number. Let \cchiless be the problem of determining whether $\cchi(G)\le r$ for a given graph $G$, and \cchiequal the problem of determining whether $\cchi(G) = r$ for a given graph $G$.

Since the problem of deciding whether $\chi' \le k$ for an integer $k\ge 3$ is equivalent to deciding whether $\cchi \le k$, both are NP-hard even for regular graphs \cite{Leven1983NP} (and it does not get any easier for multigraphs). Knowing $\chi'$ of the input graph does not reduce the complexity \cite{nadolski}: it is NP-hard to determine $\cchi$ for a given Class 2 graph, and it is NP-hard to determining whether $\cchi = \chi'$ for a given graph.

A certificate describing a $p/q$-edge-colouring of a graph $G$ is polynomial in terms of $|E(G)|$, hence \cchiless belongs to NP and is thus NP-complete.
However, to know that $\cchi(G) = r$, we also need a proof that $G$ has no $p/q$-edge-colouring for $p/q < r$. The set of candidate fractions $p/q$ is finite \cite{hackmann}, so essentially we need the proof of non-existence only for a single fraction $p/q$. This is a coNP-complete problem. This would suggest the following.

\begin{conjecture}
The problem \cchiequal is DP-complete.
\label{DPcomplete}
\end{conjecture}

Note that the situation is not the same as for SAT-UNSAT whose DP-completness is easily seen. The NP and coNP components are inherently linked in our case, as we have a specific pair of fractions for a specific graph, not an arbitrary pair of \cchiless instances. There are other known ways of determining $\cchi$ (some of them are explored in \cite{ghebleh2007theorems}); for instance, one can check all $r$-edge-colourings for the presence of tight cycles. This method could be faster in practice occasionally (a SAT instance for $41/6$-edge-colouring is much much larger than an instance for $7$-edge-colouring), yet it does not seem to remove coNP-hardness: we are proving that a $7$-edge-colouring exists, but a $7$-edge-colouring without a tight cycle does not exist.

There is a huge difference in practice when determining $\cchi$ vs. $\chi'$. The obvious culprit is the number of colours: while for integer edge-colourings it is at most $2\Delta$ even for multigraphs, it is up to $|E(G)|$ for circular edge-colourings \cite{hackmann}.
The fastest general method we have found is formulating the existence of a $p/q$-edge-colouring as an instance of SAT and then using state-of-the-art SAT solvers like Kissat \cite{kissat, sat2024} to solve it, repeating the process for all possible fractions $p/q$ (the set of fractions depends on the size of the input graph and the process can be sped up by choosing suitable fractions first; more details can be found in \cite{bernat2024circular}). Boolean variables correspond to pairs (edge, colour), and since the number of colours goes up to the number of edges potentially, the number of variables grows as a quadratic function of the number of edges. Our experience with various combinatorial problems suggests that this quadratic growth quickly leads to practical intractability with current SAT solvers, and it is the case for $\cchi$, too. The fractions with large numerator $p$ are not avoidable because even if a graph has $\cchi$ with a small numerator, one needs to prove the lower bound, i.e. exclude all smaller fractions, including those with large numerators. 

Complete sets of graphs and multigraphs with specific order were generated using nauty and genreg \cite{McKay2014,genreg}. In some instances, the numbers of such (multi)graphs were not readily available in sources like The On-Line Encyclopedia of Integer Sequences, so could be considered ``new''.

For filtering uncolourable (Vizing class 2) graphs and multigraphs, we used the CVD heuristic \cite{fiolvilaltella}, straightforward backtracking, and SAT solvers (unlike the latter two methods, the heuristic cannot prove uncolourability, but often finds a $\Delta$-edge-colouring of class 1 graphs very quickly). The approach was engineered on a case-by-case basis depending on how specific tools performed on a specific family of graphs. In certain cases, like multigraphs with $\Delta = 4$ on $13$ vertices, it proved most efficient to first filter out uncolourable graphs that contain a small subgraph forcing $\cchi = \Delta + 1$ (finding a subgraph isomorphism/monomorphism using the VF2 algorithm \cite{Cordella2004} seems much faster than determining $\cchi$).

With current computers and the mentioned computational methods, we can determine $\cchi$ of graphs with $\Delta=3$ with about $50$ edges (somewhere around this bound it starts to take several minutes per graph on a single CPU core, and at $60$ edges it might take days). However, for $\Delta = 6$, graphs with about $30$ edges typically already need hours, and for $\Delta = 7$ some graphs with just $27$ edges take weeks.

Another difficulty stems from the explosion in the number of graphs of given order for $\Delta\ge 5$. While the number of $3$-regular graphs grows about $10$x per $2$ additional vertices, it is more like $100$x per single vertex for $\Delta = 7$ \cite{meringer1999fast}, so one runs into difficulties very quickly.

Essentially, experimental results on the Upper Gap Conjecture for $\Delta \ge 7$ are currently out of reach. For $\Delta = 7$, the lower bound of the conjectured gap has denominator $7$, and a graph with $\cchi(G)=p/q$ must have a matching of size $q$, i.e. at least $14$ vertices for $q = 7$. Yet for such ``big'' graphs we can hardly determine the existence of a $p/q$-edge-colouring with $q = 2$.

\section{Overview of values of $\cchi$}

If $G$ has a loop, it does not have any circular edge-colouring. For multigraphs, the Vizing theorem only guarantees $\chi' \le \Delta + \mu$, where $\mu$ is the maximum multiplicity of an edge \cite{vizing1965}, and the Shannon bound \cite{Shannon1949} gives $\chi' \le \lfloor\frac 32 \Delta\rfloor$ (and is tight for certain graphs). We have not found any multigraphs with $\cchi > \Delta + 1$ that would contribute some interesting value of $\cchi$.

In Tables~\ref{table:maxd04_g02}, \ref{table:maxd04_g03}, \ref{table:d04_g03}, \ref{table:maxd05_g02}, \ref{table:maxd05_g03}, \ref{table:d05_g03}, \ref{table:maxd06_g02}, \ref{table:maxd06_g03}, \ref{table:d06_g03}, we only include rows for which at least one Class 2 graph of interest exists (this corresponds to a non-zero entry in the the third column).
In cases where the number of graphs is large or computation of $\cchi$ takes long, we skipped fully determining $\cchi$ of graphs from the lower half of the interval of interest (e.g. ``$\le 9/2$'' for order $15$ in Table~\ref{table:maxd04_g03}). The first occurrence of a value in a table is marked in bold.

The upper bound of $|V(G)|/2$ on denominator of $\cchi(G)$ ensures that fractions with large denominators cannot appear in the initial rows of a table. The data suggest several more empirical patterns:
\begin{itemize}
    \item A specific value of denominator does not appear immediately when the trivial bound allows it to, but only after a couple of extra vertices.
    \item For a fixed denominator, fractions with larger numerators tend to appear at larger orders than those with smaller numerators. 
    \item A specific fraction appears first for multigraphs, only later for simple graphs.
\end{itemize}
These patterns are best illustrated with denominators $3$, $4$, $5$ for $\Delta=4$.

\begin{table}
\centering
\begin{tabular}{|r|r|r|c|}
\hline
\textbf{Order} & \textbf{Multigraphs} & \textbf{With $\chi' > \Delta$} & \textbf{Values of $\cchi$} \\
\hline
3 & 5 & 3 & ${\bf 5/1}$, ${\bf 6/1}$ \\
4 & 25 & 4 & $5/1$ \\
5 & 124 & 37 & ${\bf 9/2}$, $5/1$ \\
6 & 704 & 104 & $9/2$, $5/1$ \\
7 & 4283 & 784 & ${\bf 13/3}$, $9/2$, ${\bf 14/3}$, $5/1$ \\
8 & 29773 & 3351 & $13/3$, $9/2$, $5/1$ \\
9 & 227016 & 27957 & $\bf 17/4$, $13/3$, $9/2$, $14/3$, $5/1$ \\
10 & 1916000 & 156189 & $17/4$, $13/3$, $9/2$, $14/3$, $5/1$ \\
11 & 17633748 & 1512908 & $\bf 21/5$, $17/4$, $13/3$, $\bf 22/5$, $9/2$, $14/3$, $\bf 19/4$, $5/1$ \\
12 & 176094228 & 10354895 & $21/5$, $17/4$, $13/3$, $22/5$, $9/2$, $14/3$, $19/4$, $5/1$ \\
13 & 1893482910 & 116572317 & $21/5$, $17/4$, $13/3$, $22/5$, $9/2$, $\bf 23/5$, $14/3$, $19/4$, $5/1$ \\
\hline
\end{tabular}
\caption{Multigraphs with $\Delta = 4$.}
\label{table:maxd04_g02}
\end{table}

\begin{table}
\centering
\begin{tabular}{|r|r|r|c|}
\hline
\textbf{Order} & \textbf{Graphs} & \textbf{With $\chi' > \Delta$} & \textbf{Values of $\cchi$} \\
\hline
5 & 11 & 2 & $\bf 9/2$, $\bf 5/1$ \\
6 & 49 & 2 & $5/1$ \\
7 & 289 & 15 & $9/2$, $5/1$ \\
8 & 1735 & 27 & $9/2$, $5/1$ \\
9 & 11676 & 242 & $\bf 13/3$, $9/2$, $\bf 14/3$, $5/1$ \\
10 & 87669 & 609 & $9/2$, $5/1$ \\
11 & 733811 & 7556 &  $13/3$, $9/2$, $14/3$, $5/1$ \\
12 & 6781207 & 22745 &  $13/3$, $9/2$, $5/1$ \\
13 & 68462296 & 421416 & $\bf 17/4$, $13/3$, $\bf 22/5$, $9/2$, $14/3$, $5/1$ \\
14 & 748330892 & 1344629 & $17/4$, $13/3$, $22/5$, $9/2$, $14/3$, $5/1$ \\
15 & 8787966433 & 35790315 & $\le 9/2$; $14/3$, $5/1$ \\
\hline
\end{tabular}
\caption{Simple graphs with $\Delta = 4$.}
\label{table:maxd04_g03}
\end{table}

\begin{table}
\centering
\begin{tabular}{|r|r|r|c|}
\hline
\textbf{Order} & \textbf{Graphs} & \textbf{With $\chi' > \Delta$} & \textbf{Values of $\cchi$} \\
\hline
5 & 1 & 1 & $\bf 5/1$ \\
7 & 2 & 2 & $5/1$ \\
9 & 16 & 16 & $\bf 9/2$, $\bf 14/3$, $5/1$ \\
10 & 59 & 1 & $5/1$ \\
11 & 265 & 265 & $9/2$, $14/3$, $5/1$ \\
12 & 1544 & 9 & $5/1$ \\
13 & 10778 & 10778 & $\bf 13/3$, $9/2$, $14/3$, $5/1$ \\
14 & 88168 & 197 & $9/2$, $14/3$, $5/1$ \\
15 & 805491 & 805491 & $13/3$, $\bf 22/5$, $9/2$, $14/3$, $5/1$ \\
16 & 8037418 & 6826 & $9/2$, $14/3$, $5/1$ \\
17 & 86221634 & 86221634 & $\le 9/2$; $14/3$, $5/1$ \\
18 & 985870522 & 401169 & $9/2$, $14/3$, $5/1$ \\
\hline
\end{tabular}
\caption{Simple $4$-regular graphs.}
\label{table:d04_g03}
\end{table}


\begin{table}
\centering
\begin{tabular}{|r|r|r|c|}
\hline
\textbf{Order} & \textbf{Multigraphs} & \textbf{With $\chi' > \Delta$} & \textbf{Values of $\cchi$} \\
\hline
3 & 5 & 3 & $\bf 6/1$, $\bf 7/1$ \\
4 & 56 & 12 & $6/1$, $7/1$ \\
5 & 456 & 137 & $\bf 11/2$, $6/1$, $7/1$ \\
6 & 5075 & 789 & $11/2$, $6/1$, $7/1$ \\
7 & 61840 & 10278 & $\bf 16/3$, $11/2$, $\bf 17/3$, $6/1$, $7/1$ \\
8 & 907389 & 84323 & $16/3$, $11/2$, $17/3$, $6/1$, $7/1$ \\
9 & 15346680 & 1507430 & $\bf 21/4$, $16/3$, $11/2$, $17/3$, $6/1$, $7/1$ \\
10 & 299439341 & 16082522 & $\le 11/2$; $17/3$, $\bf 23/4$, $6/1$, $7/1$\\
11 & 6646221977 & 397945318 & $\le 11/2$; $17/3$, $23/4$, $6/1$, $7/1$\\
\hline
\end{tabular}
\caption{Multigraphs with $\Delta = 5$.}
\label{table:maxd05_g02}
\end{table}

\begin{table}
\centering
\begin{tabular}{|r|r|r|c|}
\hline
\textbf{Order} & \textbf{Graphs} & \textbf{With $\chi' > \Delta$} & \textbf{Values of $\cchi$} \\
\hline
7 & 344 & 7 & $\bf 16/3$, $\bf 11/2$, $\bf 6/1$ \\
8 & 4457 & 21 & $11/2$, $6/1$ \\
9 & 63493 & 372 & $16/3$, $11/2$, $6/1$ \\
10 & 1067440 & 1586 & $11/2$, $6/1$ \\
11 & 20764005 & 65458 & $16/3$, $11/2$, $\bf 17/3$, $6/1$ \\
12 & 464341835 & 337393 & $\le 11/2$; $17/3$, $6/1$ \\
13 & 11783221545 & 54152854 & $\le 11/2$; $17/3$, $6/1$ \\
\hline
\end{tabular}
\caption{Simple graphs with $\Delta = 5$.}
\label{table:maxd05_g03}
\end{table}

\begin{table}
\centering
\begin{tabular}{|r|r|r|c|}
\hline
\textbf{Order} & \textbf{Graphs} & \textbf{With $\chi' > \Delta$} & \textbf{Values of $\cchi$} \\
\hline
14 & 3459383 & 26 & $\bf 11/2$, $\bf 6/1$ \\
16 & 2585136675 & 2885 & $11/2$, $\bf 17/3$, $\bf 23/4$, $6/1$ \\
\hline
\end{tabular}
\caption{Simple $5$-regular graphs.}
\label{table:d05_g03}
\end{table}


\begin{table}
\centering
\begin{tabular}{|r|r|r|c|}
\hline
\textbf{Order} & \textbf{Multigraphs} & \textbf{With $\chi' > \Delta$} & \textbf{Values of $\cchi$} \\
\hline
3 & 9 & 6 & $\bf 7/1$, $\bf 8/1$, $\bf 9/1$ \\
4 & 110 & 27 & $7/1$, $8/1$ \\
5 & 1524 & 492 & $\bf 13/2$, $7/1$, $\bf 15/2$, $8/1$ \\
6 & 28279 & 4392 & $13/2$, $7/1$, $8/1$ \\
7 & 649852 & 103436 & $\bf 19/3$, $13/2$, $\bf 20/3$, $7/1$, $8/1$ \\
8 & 18596557 & 1419968 & $19/3$, $13/2$, $20/3$, $7/1$, $8/1$ \\
\hline
\end{tabular}
\caption{Multigraphs with $\Delta = 6$.}
\label{table:maxd06_g02}
\end{table}

\begin{table}
\centering
\begin{tabular}{|r|r|r|c|}
\hline
\textbf{Order} & \textbf{Graphs} & \textbf{With $\chi' > \Delta$} & \textbf{Values of $\cchi$} \\
\hline
7 & 853 & 4 & $\bf 13/2$, $\bf 7/1$ \\
8 & 10073 & 12 & $13/2$, $7/1$ \\
9 & 185323 & 235 & $\le 13/2$; $\bf 20/3$, $7/1$ \\
10 & 4992383 & 1380 & $\le 13/2$; $20/3$, $7/1$ \\
11 & 160175079 & 137444 & $\le 13/2$; $20/3$, $\bf 27/4$, $7/1$ \\
\hline
\end{tabular}
\caption{Simple graphs with $\Delta = 6$.}
\label{table:maxd06_g03}
\end{table}

\begin{table}
\centering
\begin{tabular}{|r|r|r|c|}
\hline
\textbf{Order} & \textbf{Graphs} & \textbf{With $\chi' > \Delta$} & \textbf{Values of $\cchi$} \\
\hline
7 & 1 & 1 & $\bf 7/1$ \\
9 & 4 & 4 & $7/1$ \\
11 & 266 & 266 & computation takes too long \\
13 & 367860 & 367860 & computation takes too long\\
14 & 21609300 & 7 & $7/1$ \\
\hline
\end{tabular}
\caption{Simple $6$-regular graphs.}
\label{table:d06_g03}
\end{table}

\section{Graphs with large values of $\cchi$}
\label{sec:maxcchi}

Graphs constructed in proofs of Theorems~\ref{thm:7}, \ref{thm:6}, \ref{thm:5} have $\cchi = \Delta + 1$ and are thus crucial obstacles in proving any variant of the Upper Gap Conjecture. There are many other such graphs, with no obvious pattern or characterization for any $\Delta > 3$. Since they are exceptions, either their characterization must emerge from the proof, or the specific statement being proved must exclude them upfront, possibly by some stronger yet simple assumption that does not aim at characterizing all the exceptions. One suggestion for such exclusion was to assume $2$-edge-connectivity (see the discussion after Problem 1 in \cite{Kaiser2004}); our findings demonstrate that it is insufficient. Another suggestion in the same paper was edge-connectivity at least $\Delta$. This is also not sufficient, at least not for $\Delta = 4$: the circulants $C_{2n+1}(1,2)$ for $n = 2$ (this is just $K_5$) and $n = 3$ have $\cchi = 5$ \cite{bakalarka}, while they are $4$-connected.

It seems that there is no end in sight for $4$-regular graphs with $\cchi = 5$: even for $17$ and $18$ vertices we encountered such graphs, and they were $\cchi$-critical, i.e. their $\cchi$ was not forced by some smaller subgraph. The larger ones have a cut-vertex, but there is a $4$-connected one on $9$ vertices (graph6 code \verb|HEhbtjK|), too. The situation seems similar for $\Delta\in\{5, 6\}$, but for those we only managed to analyze very small graphs.

As far as we know, examples of simple non-regular graphs with $\cchi = \Delta + 1$ are known only for specific small values of maximum degree $\Delta$ (we want non-regularity to be able to use those graphs in some constructions of bigger families, possibly with large connectivity). We dare to propose a conjecture in this direction (mentioned in \cite[Conjecture 2]{bakalarka}).

\begin{conjecture}
For every even integer $d > 4$, the graph $G$ obtained from $K_{d+1}$ by removing a single edge has $\cchi(G) = d + 1$.
\end{conjecture}

We have computationally verified this conjecture for $d\in \{6,8,10,12\}$. For $d = 4$, it would hold with a small modification: after removing an edge from $K_5$, $\cchi$ drops to $9/2$, and one needs to attach a pendant vertex to one of the vertices of degree $3$ to keep $\cchi = 5$. Note that for even $d$, it is known that $\cchi(K_{d+1}) = d+1$ because $\cchi(G)\ge |E(G)|/\alpha'(G)$ \cite{Hackmann2004}. Current methods of determining $\cchi$ are not very good at dealing with arbitrary vertex degree because lower bounds are hard to establish, so conjectures of this kind might be tougher than they appear at first sight. For instance, the lower bound $\cchi(K_{d+1})\ge d+1$ is just a direct consequence of a general bound based on counting edges, paying no regard to the structure of a specific graph. Attempts to improve lower bounds above what trivial counting provides are often hampered by graphs with small value of $\cchi$ which provide a hard limit on how far a general lower bound can go \cite{macaj2013asymptotic}.



Since graphs with $\cchi > \Delta + 1/2$ seem rare, infinite classes of them are interesting, especially with higher connectivity. For graphs with $\Delta = 3$, there is only a single such graph known. In the subsequent sections, we provide some constructions of them based on generalizations of specific small examples found by the described exhaustive computer search. Unfortunately, the construction methods are very simple and only produce graphs that are not $\cchi$-critical.
We start our presentation with higher maximum degrees because somehow the situation is more complicated for $\Delta = 4$ (this might be just because it was possible to analyze more such graphs). The following lemma allows to shorten several proofs later.

\begin{lemma}
Assume that there exists a connected graph $H$ with maximum degree $\Delta$, $\cchi(H) = \Delta + 1$, exactly two vertices of degree $1$, and having no bridge except those incident with vertices of degree $1$. Then there exist infinitely many $\Delta$-regular $2$-edge-connected graphs $G$ satisfying $\cchi(G) = \Delta + 1$. In addition, if $H$ does not have a cut-vertex not adjacent to a vertex of degree $1$, then $G$ is $2$-connected.
\label{lemma:regularization}
\end{lemma}

\begin{proof}
Take $k\ge 3$ copies of $H$, label the pendant vertices of the $i$-th copy $u_i$ and $v_i$, and then identify $v_i$ with $u_{i+1}$ and suppress the resulting vertex of degree $2$ for each $i\in\{0,1,\dots, k-1\}$ (indices modulo $k$), obtaining a $2$-edge-connected graph $G'$. 
If $G'$ is regular, take $G = G'$. Otherwise, take two copies $G_1$, $G_2$ of $G'$ and repeat the following procedure for each vertex $u$ in $G_1$ that has degree $d < \Delta$: there is a vertex corresponding to $u$ in $G_2$ with the same degree, so connect them with $\Delta - d$ parallel edges. None of the added parallel edges is a bridge because $u$ must be in some copy of $H$, and other copies of $H$ have vertices corresponding to $u$, so any edge-cut separating $G_1$ from $G_2$ in the graph being constructed has at least $k\ge 3$ edges. Finally, replace each of the parallel edges with $K_{\Delta+1}$ with one edge cut into a pair of dangling edges. The resulting graph $G$ is simple, $\Delta$-regular, $2$-edge-connected, and contains $H$ as a subgraph, so $\cchi(G)\ge \Delta + 1$. The Vizing theorem guarantees $\cchi(G)\le \Delta + 1$. Our construction does not create any new cut-vertices, and the cut-vertices adjacent to vertices of degree $1$ in $H$ become part of a cycle in $G$, so there are no cut-vertices in $G$.
\end{proof}

Obviously, the ``degree regularization'' technique used in this proof is rarely efficient if we care about minimizing the number of added vertices or edges. Similar ``regularization'' approaches rarely work for small $H$ with $\cchi(H) < \Delta+1$ because edges added to vertices of $H$ tend to increase $\cchi$. In other words, even if we find a graph $H$ with some interesting value of $\cchi$ that has plenty of vertices of degree below $\Delta$, we might not be able to generalize it into an infinite class of regular graphs: we cannot cut edges of $H$ because it will decrease $\cchi$, and we cannot attach new edges to $H$ because it will increase $\cchi$. The question of what are the values of $\cchi$ that are attained, but only for a finite family of graphs, is intriguing in itself, and widely open for $\cchi > 3$  (one likely candidate is $11/3$, which is the $\cchi$ of the Petersen graph).

\section{Families with large values of $\cchi$ for $\Delta = 6$}

\begin{theorem}
There exist infinitely many graphs $G$ with $\Delta(G) = 6$ and $\cchi(G) = 6+2/3$.
\label{thm:20/3}
\end{theorem}

\begin{proof}
Among graphs on $10$ vertices, our search discovered a graph $H$ with $\cchi = 20/3$, one vertex of degree $1$, one vertex of degree $5$, and the remaining vertices of degree $6$. Its graph6 code is \verb|ICrUux}vO|.
For any $k\ge 3$, one can take a cycle of length $k$, add $k$ copies of $H$, and identify the vertex of degree $1$ in each copy with one of the vertices of the cycle (different copies paired with different vertices).
The resulting graph $G_k$ has $\cchi = 20/3$.
\end{proof}

It would be nice to have examples with higher connectivity, but we do not have any base graphs to build them from yet: even though there are vertices of degree $5$ in the discovered graphs with $\cchi = 20/3$, we cannot add edges to them because they will no longer be $20/3$-edge-colourable. However, the graph $H$ from the previous proof can be used to create graphs with $\cchi = 7$. If we attach a pendant vertex to the single vertex of degree $5$ in $H$, we get a graph $H'$ that has $\cchi(H')=7$ (verified with a computer). Application of Lemma~\ref{lemma:regularization} to $H'$ yields the following theorem.

\begin{theorem}
There exist infinitely many $2$-connected $6$-regular graphs $G$ with $\cchi(G) = 7$.
\label{thm:7}
\end{theorem}

\section{Families with large values of $\cchi$ for $\Delta = 5$}

\begin{theorem}
There exist infinitely many $2$-connected $5$-regular graphs $G$ with $\cchi(G) = 6$.
\label{thm:6}
\end{theorem}

\begin{proof}
Take the graph $H$ with graph6 code \verb|HCRUnbU| which has $\cchi(H)=6$ (verified by a computer), one vertex of degree $4$, two vertices of degree $1$, and is $2$-edge-connected and without any cut-vertices except the neighbours of vertices with degree $1$. 
Apply Lemma~\ref{lemma:regularization}.
\end{proof}

We have found $382$ multigraphs with $\cchi = 23/4$ (with $10$ or $11$ vertices). Some of them have vertex-connectivity $3$. We managed to generalize into an infinite class only one of the $68$ graphs on $10$ vertices.

\begin{theorem}
There exist infinitely many multigraphs $G$ with $\cchi(G) = 5 + 3/4$, $\Delta(G)=5$, and edge connectivity $1$.
\label{thm:23/4}
\end{theorem}

\begin{proof}
There exists a $\cchi$-critical multigraph $H$ with $\cchi(H)=23/4$ which we obtained during the above-described exhaustive search. It has sparse6 code \verb|:Ig?COaaGS?aEQ?QDL?PAe| and has a vertex of degree $1$. If we take a cycle $C$ of length $2k$ for some $k\ge 2$, we can take $2k$ copies of $H$ and identify each vertex of $C$ with the vertex of degree $1$ in exactly one of the copies of $H$. Since the colouring of the copies of $H$ can be the same and shifted in such a way that the edge incident with a vertex of $C$ is $0$, we only need colours $1$ and $2$ for the edges in $C$ (since it is even). The resulting graph thus has $\cchi = 23/4$.
\end{proof}

\begin{theorem}
There exist infinitely many $2$-connected graphs $G$ with $\Delta(G) = 5$ and $\cchi(G) = 5 + 2/3$.
\label{thm:17/3}
\end{theorem}

\begin{proof}
Take the graph $H$ with graph6 code \verb|K?BcqyYfStG?| which has $\cchi(H)=17/3$ (verified by a computer). It has two vertices $u$ and $v$ of degree $1$ and no cut vertices except the neighbours of $u$ and $v$. Let us apply a circular construction analogous to the one used in the proof of Lemma~\ref{lemma:regularization} for some even $k\ge 2$, obtaining a graph $G$. Obviously $G$ is $2$-connected and $\cchi(G)\ge 17/3$. Since the number of copies of $H$ is even, we can form disjoint pairs from neighbouring copies. In every such pair $H_1$, $H_2$, we can use any $17/3$-edge-colouring of $H_1$, and apply its ``mirror image'' to $H_2$, resulting in a colouring in which the two edges separating $H_1\cup H_2$ from the rest of $G$ have the same colour. Colourings of individual pairs can thus be pieced together and $G$ is $17/3$-edge-colourable.
\end{proof}

It is not clear to us how to create $5$-regular (multi)graphs with $\cchi = 17/3$. There are vertices of degree less than $5$ in all small graphs and multigraphs with $\cchi = 17/3$, yet typically attaching even a single edge to them increases $\cchi$ to 6.

\section{Families with large values of $\cchi$ for $\Delta = 4$}

\begin{theorem}
There exist infinitely many $2$-connected $4$-regular graphs $G$ with $\cchi(G) = 5$.
\label{thm:5}
\end{theorem}

\begin{proof}
Let $H$ be the graph obtained from $K_5$ by removing an edge. It was proved in \cite[Theorem 3]{bakalarka} that $H'$ obtained from $H$ by adding a pendant vertex attached to one of the vertices of degree $3$ has $\cchi = 5$. Apply Lemma~\ref{lemma:regularization} to $H'$.
\end{proof}

\begin{theorem}
There exist infinitely many multigraphs $G$ with $\cchi(G) = 4 + 3/4$, $\Delta(G)=4$, vertex connectivity $1$ and edge connectivity $2$.
\end{theorem}

\begin{proof}
There is a multigraph $H$ with $\cchi(H)=19/4$ (sparse6 code \verb|:LiAGWAJ?XApEOsPcL?XAv|); we obtained it from a graph found by the exhaustive search by attaching two pendant vertices and verified $\cchi$ using a computer.

We can take an arbitrary number $k\ge 2$ of copies of $H$, denote $u_i$ and $v_i$ the two vertices of degree $1$ in the $i$-th copy for each $i\in\{0,1,\dots,k-1\}$, arrange the copies in a circular fashion and add an edge $e_i = v_iu_{i+1}$ (indices modulo $k$). The resulting graph $G_k$ has $\cchi\ge 19/4$ (thanks to any of the copies of $H$) and is obviously $2$-edge-connected. It also has a $19/4$-edge-colouring. Indeed, all the copies of $H$ are $19/4$-edge-colourable, and their colourings put together form a partial colouring $\varphi$. For each $i$, the colouring $\varphi$ can be extended to $e_i$: on the colour circle of length $19/4$, the length of the longer arc between the colours of the two edges incident with $e_i$ is at least $19/8 > 2$, so there is a sufficiently long interval from which we can pick a colour for $e_i$.
\end{proof}


We are not aware of any simple graphs with $\cchi = 4 + 3/4$.

We have found $40$ simple graphs $G$ with $\cchi(G) = 14/3 = 4 + 2/3$, but all are $4$-regular and thus are not readily generalizable into infinite classes using the techniques applied elsewhere in this paper. Most (though not all) of them have a cut-vertex. If we split them at the cut-vertex (keeping a copy of the cut-vertex in each part), $\cchi$ of the parts would be at most $9/2$. More details can be found in \cite{bakalarka}, but in summary, we do not sufficiently understand why these graphs have $\cchi = 14/3$ (and only these graphs among the searched ones).

\begin{theorem}
There exist infinitely many $2$-connected $4$-regular multigraphs $G$ with $\cchi(G) = 4 + 2/3$.
\end{theorem}

\begin{proof}
Our search discovered a multigraph $H$ with sparse6 code \verb|:Hg?COoAI?QDeOhn| such that $\cchi(G) = 14/3$ and $G$ has two vertices of degree $1$ and $7$ vertices of degree $4$. Let us apply a circular construction analogous to the one used in the proof of Theorem~\ref{thm:17/3} for even $k$, obtaining a graph $G$. By a very similar argument, $G$ is $2$-connected, $4$-regular, and satisfies $\cchi(G) = 14/3$.
\end{proof}

We have discovered two $4$-regular multigraphs with $\cchi = 23/5$ (order $13$ and $14$; sparse6 codes \verb|:Lk??G`CIGPb_O`IDIUOqqEX^| and \verb|:Mk??G`CIGPb_Q`QdgOiXBLGYU^|). This is the first known value of $\cchi$ above $\Delta +1/2$ that is not in the form of $\Delta+1-1/t$.

\begin{theorem}
There exist infinitely many $4$-connected $4$-regular graphs $G$ with $\cchi(G) = 9/2$.
\label{thm:circulants}
\end{theorem}

The infinite family of graphs substantiating Theorem~\ref{thm:circulants} consists of circulants $C_{2n+1}(1, 2)$ for $n\ge 6$. After defining some helpful terms, we prove the upper and lower bound on $\cchi$ of these circulants separately in a series of lemmas.

Let $G_{2n+1} = C_{2n+1}(1, 2)$ for brevity; the vertex set of $G_{2n+1}$ is $V = \{v_0, v_1, v_2, \dots, v_{2n}\}$ and the edge set is
$$
E = \left\{\, v_iv_{i + 1},\ v_iv_{i + 2} \mid i \in \{0, 1, \dots, 2n\} \,\right\}
$$
(all indices taken modulo $2n+1$ here and throughout the rest of this section).

For $i \in \{0, 2, 4, \dots, 2n\}$, we denote $e_i = v_iv_{i+2}$, $f_i = v_{i+1}v_{i+3}$.
Thanks to the symmetry of the circulants, $f_{2n+t} = v_{2n + 1 + t}v_{2n + 3 + t} \equiv v_tv_{2+t} = e_t$ for any integer $t$.

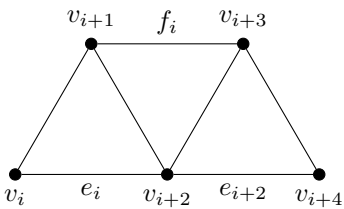
\begin{figure}[h]
    \centering
    \begin{tikzpicture}[
        vertex/.style={circle, draw, fill=black, inner sep=1.5pt},
        scale=1
    ]
    
    \node[vertex, label=below:$v_i$]     (i)   at (-2, 0) {};
    \node[vertex, label=above:$v_{i+1}$]   (i+1) at (-1, 1.732) {}; 
    \node[vertex, label=below:$v_{i+2}$]   (i+2) at (0, 0) {};
    \node[vertex, label=above:$v_{i+3}$]   (i+3) at (1, 1.732) {};
    \node[vertex, label=below:$v_{i+4}$]   (i+4) at (2, 0) {};
    
    \draw (i) -- (i+1);
    \draw (i+1) -- (i+2);
    
    \draw (i+2) -- (i+3);

    \draw (i+3) -- (i+4);
    
    \draw (i)   -- node[midway, below] {$e_i$} (i+2);
    \draw (i+1) -- node[midway, above] {$f_i$} (i+3);
    \draw (i+2) -- node[midway, below] {$e_{i+2}$} (i+4);
    
    \end{tikzpicture}
    \caption{Subgraph $H$ of \( C_{2n+1}(1,2) \) drawn with equilateral triangles.}
    \label{fig:subgraph_H}
\end{figure}

\begin{lemma}
\label{thm:aligned-bound}
Let \( \varepsilon < 1/2 \). For any valid \( (4 + \varepsilon) \)-edge-colouring of \( G_{2n+1} \) and any $i \in \{0, 2, 4, \dots, 2n\}$, one of the pairs of edges $f_i$, $e_i$ and $f_i, e_{i+2}$ has circular distance of colours at most $\varepsilon$ and the other at least $1/2$.
\end{lemma}

\begin{proof}
Consider the subgraph $H$ induced by the vertices $v_i, v_{i+1}, v_{i+2}, v_{i+3}, v_{i+4}$ (Figure~\ref{fig:subgraph_H}). In any ($4+\varepsilon)$-edge-colouring of $H$, we can shift the colours in such a way that the colour of $f_i$ is $0$. For $H$, there are altogether $24$ possible circular orderings for the colors of the edges around vertices. The ordering matters: for instance, if we know that the colour $c'$ of $v_{i+1}v_{i+2}$ is larger than the colour of $c$ of $v_iv_{i+2}$ and there is no colour of an edge incident with $v_{i+2}$ between them (i.e. $c'$ immediately follows $c$ in the ordering around $v_{i+2}$), we know that $c'\in [c+1, c+1+\varepsilon]$ (this interval represents and arc of the colour circle of length $4+\varepsilon$, so its boundaries are taken modulo $4+\varepsilon$; the interval is derived from the requirement on the circular distance between colours of incident edges). For each possible ordering, each edge gets assigned an interval describing bounds on its possible colour. This technique is a common way of analyzing possible circular edge-colourings of a small subgraph; more details can be found in \cite{ghebleh2007, mazaklukotka, mazakthesis, bakalarka}.
Half of these colourings have the interval containing colour $0$ assigned to the edge $e_i$, and the other half have it assigned to the edge $e_{i+2}$. Thanks to the symmetry of $H$, it is sufficient to analyze the first half.

These intervals are constructed systematically: we first assign an interval to each of the edges $v_{i+1}v_i$, $v_{i+1}v_{i+2}$, $v_{i+3}v_{i+2}$, $v_{i+3}v_{i+4}$, and then determine intervals for $v_iv_{i+2}$ and $v_{i+2}v_{i+4}$.  The columns in Figure~\ref{fig:subgraph colouring} are arranged according to the interval assigned to the edge $v_iv_{i+1}$.

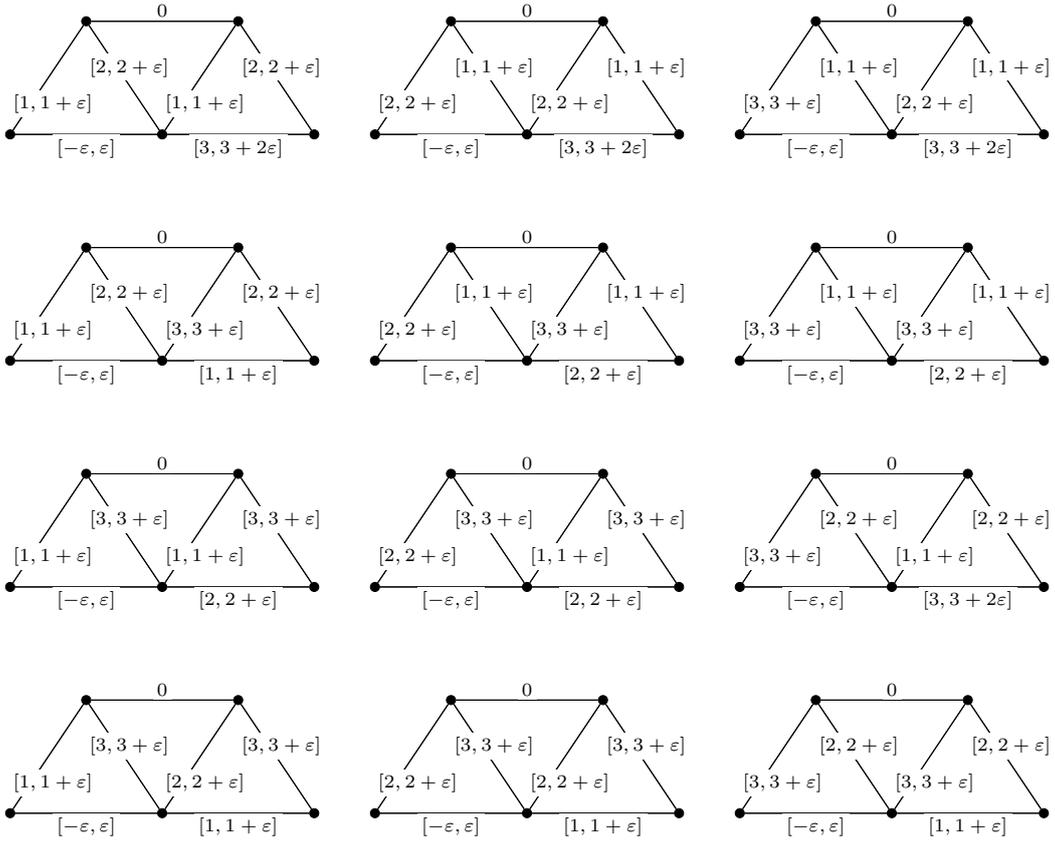
\begin{figure}[h]
    \centering
    \begin{tikzpicture}[
        vertex/.style={circle, draw, fill=black, inner sep=1.2pt},
        edge_label/.style={fill=white, inner sep=1.5pt, font=\scriptsize}
    ]


    \begin{scope}[xshift=0cm, yshift=0cm]
        \node[vertex] (v1) at (0, 1.5) {}; \node[vertex] (v3) at (2, 1.5) {};
        \node[vertex] (v0) at (-1, 0) {};  \node[vertex] (v2) at (1, 0) {};   \node[vertex] (v4) at (3, 0) {};
        \draw (v0)--(v1); \draw (v1)--(v2); \draw (v2)--(v3); \draw (v3)--(v4); \draw (v4)--(v2); \draw (v2)--(v0); \draw (v1)--(v3);
        \draw (v1) -- node[edge_label, midway, above] {$0$} (v3);
        \draw (v0) -- node[edge_label, pos=0.25, left, xshift=25pt] {$[1, 1+\varepsilon]$} (v1);
        \draw (v1) -- node[edge_label, pos=0.4, right, xshift=-12pt] {$[2, 2+\varepsilon]$} (v2);
        \draw (v2) -- node[edge_label, pos=0.25, left, xshift=25pt] {$[1, 1+\varepsilon]$} (v3);
        \draw (v3) -- node[edge_label, pos=0.4, right, xshift=-12pt] {$[2, 2+\varepsilon]$} (v4);
        \draw (v0) -- node[edge_label, midway, below] {$[-\varepsilon, \varepsilon]$} (v2);
        \draw (v2) -- node[edge_label, midway, below] {$[3, 3+2\varepsilon]$} (v4);
    \end{scope}

    \begin{scope}[xshift=4.8cm, yshift=0cm]
        \node[vertex] (v1) at (0, 1.5) {}; \node[vertex] (v3) at (2, 1.5) {};
        \node[vertex] (v0) at (-1, 0) {};  \node[vertex] (v2) at (1, 0) {};   \node[vertex] (v4) at (3, 0) {};
        \draw (v0)--(v1); \draw (v1)--(v2); \draw (v2)--(v3); \draw (v3)--(v4); \draw (v4)--(v2); \draw (v2)--(v0); \draw (v1)--(v3);
        \draw (v1) -- node[edge_label, midway, above] {$0$} (v3);
        \draw (v0) -- node[edge_label, pos=0.25, left, xshift=25pt] {$[2, 2+\varepsilon]$} (v1);
        \draw (v1) -- node[edge_label, pos=0.4, right, xshift=-12pt] {$[1, 1+\varepsilon]$} (v2);
        \draw (v2) -- node[edge_label, pos=0.25, left, xshift=25pt] {$[2, 2+\varepsilon]$} (v3);
        \draw (v3) -- node[edge_label, pos=0.4, right, xshift=-12pt] {$[1, 1+\varepsilon]$} (v4);
        \draw (v0) -- node[edge_label, midway, below] {$[-\varepsilon, \varepsilon]$} (v2);
        \draw (v2) -- node[edge_label, midway, below] {$[3, 3+2\varepsilon]$} (v4);
    \end{scope}

    \begin{scope}[xshift=9.6cm, yshift=0cm]
        \node[vertex] (v1) at (0, 1.5) {}; \node[vertex] (v3) at (2, 1.5) {};
        \node[vertex] (v0) at (-1, 0) {};  \node[vertex] (v2) at (1, 0) {};   \node[vertex] (v4) at (3, 0) {};
        \draw (v0)--(v1); \draw (v1)--(v2); \draw (v2)--(v3); \draw (v3)--(v4); \draw (v4)--(v2); \draw (v2)--(v0); \draw (v1)--(v3);
        \draw (v1) -- node[edge_label, midway, above] {$0$} (v3);
        \draw (v0) -- node[edge_label, pos=0.25, left, xshift=25pt] {$[3, 3+\varepsilon]$} (v1);
        \draw (v1) -- node[edge_label, pos=0.4, right, xshift=-12pt] {$[1, 1+\varepsilon]$} (v2);
        \draw (v2) -- node[edge_label, pos=0.25, left, xshift=25pt] {$[2, 2+\varepsilon]$} (v3);
        \draw (v3) -- node[edge_label, pos=0.4, right, xshift=-12pt] {$[1, 1+\varepsilon]$} (v4);
        \draw (v0) -- node[edge_label, midway, below] {$[-\varepsilon, \varepsilon]$} (v2);
        \draw (v2) -- node[edge_label, midway, below] {$[3, 3+2\varepsilon]$} (v4);
    \end{scope}

    
    \begin{scope}[xshift=0cm, yshift=-3.0cm]
        \node[vertex] (v1) at (0, 1.5) {}; \node[vertex] (v3) at (2, 1.5) {};
        \node[vertex] (v0) at (-1, 0) {};  \node[vertex] (v2) at (1, 0) {};   \node[vertex] (v4) at (3, 0) {};
        \draw (v0)--(v1); \draw (v1)--(v2); \draw (v2)--(v3); \draw (v3)--(v4); \draw (v4)--(v2); \draw (v2)--(v0); \draw (v1)--(v3);
        \draw (v1) -- node[edge_label, midway, above] {$0$} (v3);
        \draw (v0) -- node[edge_label, pos=0.25, left, xshift=25pt] {$[1, 1+\varepsilon]$} (v1);
        \draw (v1) -- node[edge_label, pos=0.4, right, xshift=-12pt] {$[2, 2+\varepsilon]$} (v2);
        \draw (v2) -- node[edge_label, pos=0.25, left, xshift=25pt] {$[3, 3+\varepsilon]$} (v3);
        \draw (v3) -- node[edge_label, pos=0.4, right, xshift=-12pt] {$[2, 2+\varepsilon]$} (v4);
        \draw (v0) -- node[edge_label, midway, below] {$[-\varepsilon, \varepsilon]$} (v2);
        \draw (v2) -- node[edge_label, midway, below] {$[1, 1+\varepsilon]$} (v4);
    \end{scope}

    \begin{scope}[xshift=4.8cm, yshift=-3.0cm]
        \node[vertex] (v1) at (0, 1.5) {}; \node[vertex] (v3) at (2, 1.5) {};
        \node[vertex] (v0) at (-1, 0) {};  \node[vertex] (v2) at (1, 0) {};   \node[vertex] (v4) at (3, 0) {};
        \draw (v0)--(v1); \draw (v1)--(v2); \draw (v2)--(v3); \draw (v3)--(v4); \draw (v4)--(v2); \draw (v2)--(v0); \draw (v1)--(v3);
        \draw (v1) -- node[edge_label, midway, above] {$0$} (v3);
        \draw (v0) -- node[edge_label, pos=0.25, left, xshift=25pt] {$[2, 2+\varepsilon]$} (v1);
        \draw (v1) -- node[edge_label, pos=0.4, right, xshift=-12pt] {$[1, 1+\varepsilon]$} (v2);
        \draw (v2) -- node[edge_label, pos=0.25, left, xshift=25pt] {$[3, 3+\varepsilon]$} (v3);
        \draw (v3) -- node[edge_label, pos=0.4, right, xshift=-12pt] {$[1, 1+\varepsilon]$} (v4);
        \draw (v0) -- node[edge_label, midway, below] {$[-\varepsilon, \varepsilon]$} (v2);
        \draw (v2) -- node[edge_label, midway, below] {$[2, 2+\varepsilon]$} (v4);
    \end{scope}

    \begin{scope}[xshift=9.6cm, yshift=-3.0cm]
        \node[vertex] (v1) at (0, 1.5) {}; \node[vertex] (v3) at (2, 1.5) {};
        \node[vertex] (v0) at (-1, 0) {};  \node[vertex] (v2) at (1, 0) {};   \node[vertex] (v4) at (3, 0) {};
        \draw (v0)--(v1); \draw (v1)--(v2); \draw (v2)--(v3); \draw (v3)--(v4); \draw (v4)--(v2); \draw (v2)--(v0); \draw (v1)--(v3);
        \draw (v1) -- node[edge_label, midway, above] {$0$} (v3);
        \draw (v0) -- node[edge_label, pos=0.25, left, xshift=25pt] {$[3, 3+\varepsilon]$} (v1);
        \draw (v1) -- node[edge_label, pos=0.4, right, xshift=-12pt] {$[1, 1+\varepsilon]$} (v2);
        \draw (v2) -- node[edge_label, pos=0.25, left, xshift=25pt] {$[3, 3+\varepsilon]$} (v3);
        \draw (v3) -- node[edge_label, pos=0.4, right, xshift=-12pt] {$[1, 1+\varepsilon]$} (v4);
        \draw (v0) -- node[edge_label, midway, below] {$[-\varepsilon, \varepsilon]$} (v2);
        \draw (v2) -- node[edge_label, midway, below] {$[2, 2+\varepsilon]$} (v4);
    \end{scope}


    \begin{scope}[xshift=0cm, yshift=-6.0cm]
        \node[vertex] (v1) at (0, 1.5) {}; \node[vertex] (v3) at (2, 1.5) {};
        \node[vertex] (v0) at (-1, 0) {};  \node[vertex] (v2) at (1, 0) {};   \node[vertex] (v4) at (3, 0) {};
        \draw (v0)--(v1); \draw (v1)--(v2); \draw (v2)--(v3); \draw (v3)--(v4); \draw (v4)--(v2); \draw (v2)--(v0); \draw (v1)--(v3);
        \draw (v1) -- node[edge_label, midway, above] {$0$} (v3);
        \draw (v0) -- node[edge_label, pos=0.25, left, xshift=25pt] {$[1, 1+\varepsilon]$} (v1);
        \draw (v1) -- node[edge_label, pos=0.4, right, xshift=-12pt] {$[3, 3+\varepsilon]$} (v2);
        \draw (v2) -- node[edge_label, pos=0.25, left, xshift=25pt] {$[1, 1+\varepsilon]$} (v3);
        \draw (v3) -- node[edge_label, pos=0.4, right, xshift=-12pt] {$[3, 3+\varepsilon]$} (v4);
        \draw (v0) -- node[edge_label, midway, below] {$[-\varepsilon, \varepsilon]$} (v2);
        \draw (v2) -- node[edge_label, midway, below] {$[2, 2+\varepsilon]$} (v4);
    \end{scope}

    \begin{scope}[xshift=4.8cm, yshift=-6.0cm]
        \node[vertex] (v1) at (0, 1.5) {}; \node[vertex] (v3) at (2, 1.5) {};
        \node[vertex] (v0) at (-1, 0) {};  \node[vertex] (v2) at (1, 0) {};   \node[vertex] (v4) at (3, 0) {};
        \draw (v0)--(v1); \draw (v1)--(v2); \draw (v2)--(v3); \draw (v3)--(v4); \draw (v4)--(v2); \draw (v2)--(v0); \draw (v1)--(v3);
        \draw (v1) -- node[edge_label, midway, above] {$0$} (v3);
        \draw (v0) -- node[edge_label, pos=0.25, left, xshift=25pt] {$[2, 2+\varepsilon]$} (v1);
        \draw (v1) -- node[edge_label, pos=0.4, right, xshift=-12pt] {$[3, 3+\varepsilon]$} (v2);
        \draw (v2) -- node[edge_label, pos=0.25, left, xshift=25pt] {$[1, 1+\varepsilon]$} (v3);
        \draw (v3) -- node[edge_label, pos=0.4, right, xshift=-12pt] {$[3, 3+\varepsilon]$} (v4);
        \draw (v0) -- node[edge_label, midway, below] {$[-\varepsilon, \varepsilon]$} (v2);
        \draw (v2) -- node[edge_label, midway, below] {$[2, 2+\varepsilon]$} (v4);
    \end{scope}

    \begin{scope}[xshift=9.6cm, yshift=-6.0cm]
        \node[vertex] (v1) at (0, 1.5) {}; \node[vertex] (v3) at (2, 1.5) {};
        \node[vertex] (v0) at (-1, 0) {};  \node[vertex] (v2) at (1, 0) {};   \node[vertex] (v4) at (3, 0) {};
        \draw (v0)--(v1); \draw (v1)--(v2); \draw (v2)--(v3); \draw (v3)--(v4); \draw (v4)--(v2); \draw (v2)--(v0); \draw (v1)--(v3);
        \draw (v1) -- node[edge_label, midway, above] {$0$} (v3);
        \draw (v0) -- node[edge_label, pos=0.25, left, xshift=25pt] {$[3, 3+\varepsilon]$} (v1);
        \draw (v1) -- node[edge_label, pos=0.4, right, xshift=-12pt] {$[2, 2+\varepsilon]$} (v2);
        \draw (v2) -- node[edge_label, pos=0.25, left, xshift=25pt] {$[1, 1+\varepsilon]$} (v3);
        \draw (v3) -- node[edge_label, pos=0.4, right, xshift=-12pt] {$[2, 2+\varepsilon]$} (v4);
        \draw (v0) -- node[edge_label, midway, below] {$[-\varepsilon, \varepsilon]$} (v2);
        \draw (v2) -- node[edge_label, midway, below] {$[3, 3+2\varepsilon]$} (v4);
    \end{scope}
    
    
    \begin{scope}[xshift=0cm, yshift=-9.0cm]
        \node[vertex] (v1) at (0, 1.5) {}; \node[vertex] (v3) at (2, 1.5) {};
        \node[vertex] (v0) at (-1, 0) {};  \node[vertex] (v2) at (1, 0) {};   \node[vertex] (v4) at (3, 0) {};
        \draw (v0)--(v1); \draw (v1)--(v2); \draw (v2)--(v3); \draw (v3)--(v4); \draw (v4)--(v2); \draw (v2)--(v0); \draw (v1)--(v3);
        \draw (v1) -- node[edge_label, midway, above] {$0$} (v3);
        \draw (v0) -- node[edge_label, pos=0.25, left, xshift=25pt] {$[1, 1+\varepsilon]$} (v1);
        \draw (v1) -- node[edge_label, pos=0.4, right, xshift=-12pt] {$[3, 3+\varepsilon]$} (v2);
        \draw (v2) -- node[edge_label, pos=0.25, left, xshift=25pt] {$[2, 2+\varepsilon]$} (v3);
        \draw (v3) -- node[edge_label, pos=0.4, right, xshift=-12pt] {$[3, 3+\varepsilon]$} (v4);
        \draw (v0) -- node[edge_label, midway, below] {$[-\varepsilon, \varepsilon]$} (v2);
        \draw (v2) -- node[edge_label, midway, below] {$[1, 1+\varepsilon]$} (v4);
    \end{scope}

    \begin{scope}[xshift=4.8cm, yshift=-9.0cm]
        \node[vertex] (v1) at (0, 1.5) {}; \node[vertex] (v3) at (2, 1.5) {};
        \node[vertex] (v0) at (-1, 0) {};  \node[vertex] (v2) at (1, 0) {};   \node[vertex] (v4) at (3, 0) {};
        \draw (v0)--(v1); \draw (v1)--(v2); \draw (v2)--(v3); \draw (v3)--(v4); \draw (v4)--(v2); \draw (v2)--(v0); \draw (v1)--(v3);
        \draw (v1) -- node[edge_label, midway, above] {$0$} (v3);
        \draw (v0) -- node[edge_label, pos=0.25, left, xshift=25pt] {$[2, 2+\varepsilon]$} (v1);
        \draw (v1) -- node[edge_label, pos=0.4, right, xshift=-12pt] {$[3, 3+\varepsilon]$} (v2);
        \draw (v2) -- node[edge_label, pos=0.25, left, xshift=25pt] {$[2, 2+\varepsilon]$} (v3);
        \draw (v3) -- node[edge_label, pos=0.4, right, xshift=-12pt] {$[3, 3+\varepsilon]$} (v4);
        \draw (v0) -- node[edge_label, midway, below] {$[-\varepsilon, \varepsilon]$} (v2);
        \draw (v2) -- node[edge_label, midway, below] {$[1, 1+\varepsilon]$} (v4);
    \end{scope}

    \begin{scope}[xshift=9.6cm, yshift=-9.0cm]
        \node[vertex] (v1) at (0, 1.5) {}; \node[vertex] (v3) at (2, 1.5) {};
        \node[vertex] (v0) at (-1, 0) {};  \node[vertex] (v2) at (1, 0) {};   \node[vertex] (v4) at (3, 0) {};
        \draw (v0)--(v1); \draw (v1)--(v2); \draw (v2)--(v3); \draw (v3)--(v4); \draw (v4)--(v2); \draw (v2)--(v0); \draw (v1)--(v3);
        \draw (v1) -- node[edge_label, midway, above] {$0$} (v3);
        \draw (v0) -- node[edge_label, pos=0.25, left, xshift=25pt] {$[3, 3+\varepsilon]$} (v1);
        \draw (v1) -- node[edge_label, pos=0.4, right, xshift=-12pt] {$[2, 2+\varepsilon]$} (v2);
        \draw (v2) -- node[edge_label, pos=0.25, left, xshift=25pt] {$[3, 3+\varepsilon]$} (v3);
        \draw (v3) -- node[edge_label, pos=0.4, right, xshift=-12pt] {$[2, 2+\varepsilon]$} (v4);
        \draw (v0) -- node[edge_label, midway, below] {$[-\varepsilon, \varepsilon]$} (v2);
        \draw (v2) -- node[edge_label, midway, below] {$[1, 1+\varepsilon]$} (v4);
    \end{scope}

    \end{tikzpicture}
    \caption{All $12$ possible colourings of $H$.}
    \label{fig:subgraph colouring}
\end{figure}

In all the depicted configurations, the circular distance between the colours of $f_i$ and $e_i$ is at most $\varepsilon$ and the circular distance of colours of $f_i$ and $e_{i+2}$ is more than $1/2$ (for instance, for the top left corner case, $0$ is equal to $4+\varepsilon$ on the colour circle, so its circular distance from the colour of $e_{i+2}$ is at least $(4+\varepsilon)-(3+2\varepsilon) = 1-\varepsilon > 1/2$).
\end{proof}

Since $e_t \equiv f_{2n+t}$, the property of edges $f_i$ proved in Lemma~\ref{thm:aligned-bound} also holds for edges $e_i$. Given a fixed $(4+\varepsilon)$-edge-colouring of $G_{2n+1}$, we say that a pair of edges $e_i$, $f_i$ is aligned if their colours differ by less than $\varepsilon$.

\begin{lemma}
\label{thm:lower-bound}
For every positive integer \( n \), $\displaystyle \cchi\left(G_{2n+1}\right) \geq 4 + \frac{1}{2}$.
\end{lemma}

\begin{proof}
Assume that $G_{2n+1}$ admits a circular $(4+\varepsilon)$-edge-colouring $\varphi$ for some $\varepsilon < 1/2$. We will derive a contradiction.

According to Lemma~\ref{thm:aligned-bound}, $f_0$ is aligned either with $e_0$ or with $e_2$. Thanks to the symmetry of $G_{2n+1}$, we only need to analyze the case where $f_0$ is aligned with $e_0$. Then it is not aligned with $e_2$. Since $e_2$ has to be aligned with one of $f_0$, $f_2$, it has to be aligned with $f_2$. We can extend the alignment of $e_i$ with $f_i$ by induction to all $i \in \{0,2,\dots,2n\}$. Consequently, $e_{2n}$ is aligned with $f_{2n} \equiv e_0$ (see Figure~\ref{fig:circulant_contradiction}). Hence
$$
|\varphi(e_{2n}) - \varphi(e_0)| \le \varepsilon.
$$
Now we can use the triangle inequality to derive a contraction: the edges $e_{2n}$ and $f_0$ are both incident with $v_1$, so 
$$
1\le |\varphi(e_{2n}) - \varphi(f_0)| \leq |\varphi(e_{2n}) - \varphi(e_0)| + |\varphi(e_0) - \varphi(f_0)| \le 2\varepsilon < 1.
$$
\end{proof}

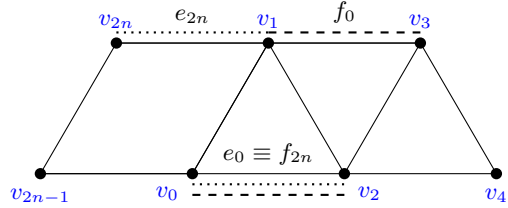
\begin{figure}
    \centering
    \begin{tikzpicture}[
        vertex/.style={circle, fill=black, inner sep=1.5pt, outer sep=0pt},
        vlabel/.style={text=blue, font=\small},
        edge_label/.style={midway, font=\small, text=black}
    ]
    \usetikzlibrary{calc}

    \def\h{1.732}

    \coordinate (v2n-1) at (-2,0);
    \coordinate (v0) at (0,0);
    \coordinate (v2) at (2,0);
    \coordinate (v4) at (4,0);

    \coordinate (v2n) at (-1,\h);
    \coordinate (v1) at (1,\h);
    \coordinate (v3) at (3,\h);

    \draw (v2n-1) -- (v2n) -- (v1) -- (v0) -- (v2n-1);
    \draw (v0) -- (v1) -- (v2) -- (v3) -- (v1);
    \draw (v2) -- (v4);
    \draw (v2n-1) -- (v0);
    \draw (v2n) -- (v3);
    \draw (v3) -- (v4);
    
    \draw (v0) -- (v2) node[edge_label, above] {$e_0 \equiv f_{2n}$};

    \node[vertex, label={[vlabel]below:$v_{2n-1}$}] at (v2n-1) {};
    \node[vertex, label={[vlabel]below left:$v_0$}] at (v0) {};
    \node[vertex, label={[vlabel]below right:$v_2$}] at (v2) {};
    \node[vertex, label={[vlabel]below:$v_4$}] at (v4) {};

    \node[vertex, label={[vlabel]above:$v_{2n}$}] at (v2n) {};
    \node[vertex, label={[vlabel]above:$v_1$}] at (v1) {};
    \node[vertex, label={[vlabel]above:$v_3$}] at (v3) {};

    \def\offset{4pt} 

    \draw[dotted, thick] ([yshift=\offset]v2n) -- ([yshift=\offset]v1) node[edge_label, above] {$e_{2n}$};

    \draw[dashed, thick] ([yshift=\offset]v1) -- ([yshift=\offset]v3) node[edge_label, above] {$f_0$};

    \draw[dotted, thick] ([yshift=-\offset]v0) -- ([yshift=-\offset]v2);
    \draw[dashed, thick] ([yshift=-2*\offset]v0) -- ([yshift=-2*\offset]v2);

    \end{tikzpicture}
    \caption{Contradiction for $\varepsilon < 1/2$ in $C_{2n+1}(1,2)$.}
    \label{fig:circulant_contradiction}
\end{figure}

With a computer, we determined $\cchi(G_5) = \cchi(G_7) = 5$ and $\cchi(G_9)=\cchi(G_{11}) = 14/3$. All the larger circulants have a $9/2$-edge-colouring.

\begin{lemma}
For $n\ge 6$, $G_{2n+1}$ is $(9,2)$-edge-colourable.
\end{lemma}

\begin{proof}
Figures~\ref{fig:G13} and \ref{fig:G15} show a $(9,2)$-edge-colouring of $G_{13}$ and $G_{15}$, respectively (the dangling edges wrap around). Figure~\ref{fig:segment} shows a colouring of a segment with $4$ vertices that has the same colours on dangling edges as the previously given colourings. Since every odd integer $m\ge 13$ can be written as $13+4k$ or $15+4k$ for a suitable $k$, the depicted colourings can be combined into a valid $(9,2)$-edge-colouring of any such $G_m$.
\end{proof}

\begin{figure}[h!]
    \centering
    \begin{tikzpicture}[
        vertex/.style={circle, fill=black, inner sep=2pt, outer sep=0pt},
        elabel/.style={font=\small\color{black}, midway, fill=white, inner sep=1pt},
        dist1_edge/.style={bend right=0},
        dist2_edge/.style={bend left=50}
    ]

    \foreach \i in {0,...,12} {
        \node[vertex] (v\i) at (\i * 0.9, 0) {};
    }

    
    \def\distOneColors{{2, 0, 5, 0, 7, 1, 8, 3, 8, 3, 8, 2}} 
    \def\distTwoColors{{7, 7, 2, 3, 4, 5, 6, 1, 1, 6, 6}}    

    \def\colorRightTwelveZero{4}  
    \def\colorRightElevenZero{4}  
    \def\colorRightTwelveOne{0}   
    
    \def\colorLeftOneTwelve{4}    
    \def\colorLeftZeroTwelve{4}   
    \def\colorLeftZeroEleven{0}   

    \foreach \i in {0,...,11} {
        \pgfmathtruncatemacro{\nexti}{\i+1}
        \draw (v\i) to[dist1_edge] node[elabel, below = 3 pt] {\pgfmathparse{\distOneColors[\i]}\pgfmathresult} (v\nexti);
    }

    \foreach \i in {0,...,10} {
        \pgfmathtruncatemacro{\nexti}{\i+2}
        \draw (v\i) to[dist2_edge] node[elabel, above = 3 pt] {\pgfmathparse{\distTwoColors[\i]}\pgfmathresult} (v\nexti);
    }

    \def\xRightAlign{12}
    \def\xLeftAlign{-1}
    
    \draw (v12) to[out=60, in=180] (\xRightAlign, 0.8) node[font=\small\color{blue}, right=2pt] {\colorRightTwelveZero};
    \draw (v11) to[out=45, in=180]  (\xRightAlign, 0.4)   node[font=\small\color{blue}, right=2pt] {\colorRightElevenZero};
    \draw (v12) to[out=0, in=180] (\xRightAlign, 0) node[font=\small\color{blue}, right=2pt] {\colorRightTwelveOne};

    \draw (v1) to[out=135, in=0] (\xLeftAlign, 0.8) node[font=\small\color{blue}, left=2pt] {\colorLeftOneTwelve};
    \draw (v0) to[out=180, in=0] (\xLeftAlign, 0)   node[font=\small\color{blue}, left=2pt] {\colorLeftZeroEleven};
    \draw (v0) to[out=135, in=0] (\xLeftAlign, 0.4) node[font=\small\color{blue}, left=2pt] {\colorLeftZeroTwelve};

    \end{tikzpicture}
    \caption{$(9,2)$-edge-colouring of \( G_{13} \)}
    \label{fig:G13}
\end{figure}
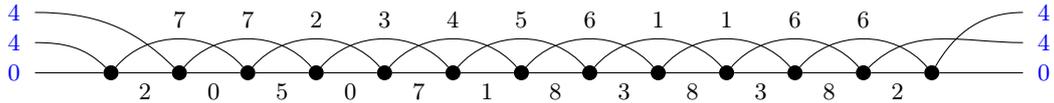

\begin{figure}[h!]
    \centering
    \begin{tikzpicture}[
        vertex/.style={circle, fill=black, inner sep=2pt, outer sep=0pt},
        elabel/.style={font=\small\color{black}, midway, fill=white, inner sep=1pt},
        dist1_edge/.style={bend right=0},
        dist2_edge/.style={bend left=50}
    ]

    \foreach \i in {0,...,14} {
        \node[vertex] (v\i) at (\i * 0.9, 0) {};
    }

    
    \def\distOneColors{{2, 0, 5, 0, 7, 1, 8, 3, 8, 6, 8, 2, 0, 2}} 
    \def\distTwoColors{{7, 7, 2, 3, 4, 5, 6, 1, 1, 4, 4, 6, 6}}    

    \def\colorRightTwelveZero{4}  
    \def\colorRightElevenZero{4}  
    \def\colorRightTwelveOne{0}   
    
    \def\colorLeftOneTwelve{4}    
    \def\colorLeftZeroTwelve{4}   
    \def\colorLeftZeroEleven{0}   

    \foreach \i in {0,...,13} {
        \pgfmathtruncatemacro{\nexti}{\i+1}
        \draw (v\i) to[dist1_edge] node[elabel, below = 3 pt] {\pgfmathparse{\distOneColors[\i]}\pgfmathresult} (v\nexti);
    }

    \foreach \i in {0,...,12} {
        \pgfmathtruncatemacro{\nexti}{\i+2}
        \draw (v\i) to[dist2_edge] node[elabel, above = 3 pt] {\pgfmathparse{\distTwoColors[\i]}\pgfmathresult} (v\nexti);
    }

    \def\xRightAlign{13.5}
    \def\xLeftAlign{-.7}
    
    \draw (v14) to[out=60, in=180] (\xRightAlign, 0.8) node[font=\small\color{blue}, right=2pt] {\colorRightTwelveZero};
    \draw (v13) to[out=45, in=180]  (\xRightAlign, 0.4)   node[font=\small\color{blue}, right=2pt] {\colorRightElevenZero};
    \draw (v14) to[out=0, in=180] (\xRightAlign, 0) node[font=\small\color{blue}, right=2pt] {\colorRightTwelveOne};

    \draw (v1) to[out=135, in=0] (\xLeftAlign, 0.8) node[font=\small\color{blue}, left=2pt] {\colorLeftOneTwelve};
    \draw (v0) to[out=180, in=0] (\xLeftAlign, 0)   node[font=\small\color{blue}, left=2pt] {\colorLeftZeroEleven};
    \draw (v0) to[out=135, in=0] (\xLeftAlign, 0.4) node[font=\small\color{blue}, left=2pt] {\colorLeftZeroTwelve};

    \end{tikzpicture}
    \caption{$(9,2)$-edge-colouring of \( G_{15} \)}
    \label{fig:G15}
\end{figure}
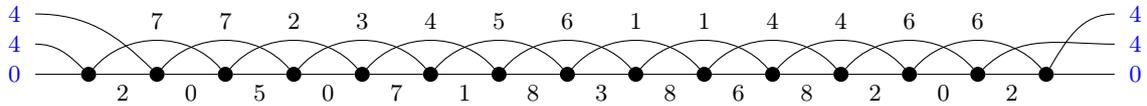

\begin{figure}[h!]
    \centering
    \begin{tikzpicture}[
        vertex/.style={circle, fill=black, inner sep=2pt, outer sep=0pt},
        elabel/.style={font=\small\color{black}, midway, fill=white, inner sep=1pt},
        dist1_edge/.style={bend right=0},
        dist2_edge/.style={bend left=50}
    ]

    \foreach \i in {0,...,3} {
        \node[vertex] (v\i) at (\i * 0.9, 0) {};
    }

    
    \def\distOneColors{{2, 0, 2}} 
    \def\distTwoColors{{6, 6}}    

    \def\colorRightTwelveZero{4}  
    \def\colorRightElevenZero{4}  
    \def\colorRightTwelveOne{0}   
    
    \def\colorLeftOneTwelve{4}    
    \def\colorLeftZeroTwelve{4}   
    \def\colorLeftZeroEleven{0}   

    \foreach \i in {0,...,2} {
        \pgfmathtruncatemacro{\nexti}{\i+1}
        \draw (v\i) to[dist1_edge] node[elabel, below = 3pt] {\pgfmathparse{\distOneColors[\i]}\pgfmathresult} (v\nexti);
    }

    \foreach \i in {0,...,1} {
        \pgfmathtruncatemacro{\nexti}{\i+2}
        \draw (v\i) to[dist2_edge] node[elabel, above = 3 pt] {\pgfmathparse{\distTwoColors[\i]}\pgfmathresult} (v\nexti);
    }

    \def\xRightAlign{4}
    \def\xLeftAlign{-1}
    
    \draw (v3) to[out=60, in=180] (\xRightAlign, 0.8) node[font=\small\color{blue}, right=2pt] {\colorRightTwelveZero};
    \draw (v2) to[out=45, in=180]  (\xRightAlign, 0.4)   node[font=\small\color{blue}, right=2pt] {\colorRightElevenZero};
    \draw (v3) to[out=0, in=180] (\xRightAlign, 0) node[font=\small\color{blue}, right=2pt] {\colorRightTwelveOne};

    \draw (v1) to[out=135, in=0] (\xLeftAlign, 0.8) node[font=\small\color{blue}, left=2pt] {\colorLeftOneTwelve};
    \draw (v0) to[out=180, in=0] (\xLeftAlign, 0)   node[font=\small\color{blue}, left=2pt] {\colorLeftZeroEleven};
    \draw (v0) to[out=135, in=0] (\xLeftAlign, 0.4) node[font=\small\color{blue}, left=2pt] {\colorLeftZeroTwelve};

    \end{tikzpicture}
    \caption{$(9,2)$-edge-colouring of a $4$-vertex segment of a circulant}
    \label{fig:segment}
\end{figure}
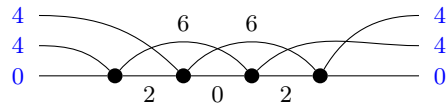

\section{Open problems}

Some of these problems were touched upon before in this paper, but we list them here for easier future referencing.

\begin{problem}
Is it true that if $\cchi(G)\in (9/2, 5]$ for a simple graph $G$ with $\Delta(G) = 4$, then $\cchi(G)\in\{14/3, 5\}$? (And if $\cchi(G) = 14/3$, then $G$ is $4$-regular?)
\end{problem}

\begin{problem}
For which $r$ does there exist a (multi)graph with $\cchi = r$, yet there are only finitely many ($\cchi$-critical?) (multi)graphs $G$ with $\cchi(G) = r$?
\end{problem}

\begin{problem}
For which $r$ does there exist a multigraph with $\cchi = r$, yet there are no simple graphs with $\cchi = r$?
\end{problem}

\begin{problem}
For which $r$ does there exist a (multi)graph with $\cchi = r$ and maximum degree $\Delta$, yet there are no $\Delta$-regular (multi)graphs with $\cchi = r$?
\end{problem}

Some multigraphs have $\cchi$ greater than $\Delta + 1$. It is not known whether such multigraphs contribute some unique values of $\cchi$ that are not attained for multigraphs with higher maximum degree.

\begin{problem}
Does there exist $r$ such that all multigraphs $G$ with $\cchi(G)=r$ have $\Delta(G) + 1 < r$?
\end{problem}

It has been known for a long time that one can create $3$-connected $3$-regular graphs $G$ with $\cchi(G) = 3+1/2$ by replacing a vertex of any $3$-regular graph with a copy of the Petersen graph with one vertex removed \cite{mazaklukotka}. Theorem~\ref{thm:circulants} proves the existence of a $4$-connected $4$-regular graph $G$ with $\cchi(G) = 4+1/2$ (an infinite family of them, actually).

\begin{problem}
For which $d$ does there exist a $d$-connected $d$-regular graph with circular chromatic index at least $d+1/2$?
\end{problem}

It would also be interesting to know which denominators are possible in the upper half of the interval $[\Delta, \Delta + 1]$.

\begin{problem}
Does there for every integer $q > 2$ exist an integer $p$ and a (multi)graph $G$ with $\cchi(G) = p/q \in (\Delta(G) + 1/2, \Delta(G)+1)$?
\end{problem}

Currently, for any $\Delta\ge 8$, we don't even know if there is a single graph with maximum degree $\Delta$ and $\cchi\in (\Delta(G) + 1/2, \Delta(G)+1)$. 

\section*{Acknowledgements}

We thank Xuding Zhu for discussions on various topics related to circular 
colourings.

\bigskip

This work was partially supported from the research grants APVV-23-0076, VEGA 1/0173/25, and VEGA 1/0613/26.

\bigskip

During the preparation of this work, the authors used Anthropic's Claude and Google's Gemini in order to generate and check some parts of the computer code, draw figures, check the article text for mistakes, and to obtain suggestions to improve clarity of the text. After using these tools/services, the authors reviewed and edited the content as needed and take full responsibility for the content of the published article.

\bibliographystyle{plain}

\bibliography{literature}

@article{Lin2015,
  title   = {Circular chromatic indices of even degree regular graphs},
  journal = {Discrete Mathematics},
  volume  = {338},
  number  = {7},
  pages   = {1154-1162},
  year    = {2015},
  issn    = {0012-365X},
  doi     = {10.1016/j.disc.2015.01.037},
  author  = {Cheyu Lin and Tsai-Lien Wong and Xuding Zhu},
}

@article{Lin2013,
  title   = {Circular Chromatic Indices of Regular Graphs},
  journal = {Journal of Graph Theory},
  volume  = {76},
  number  = {3},
  pages   = {169-193},
  year    = {2014},
  doi     = {10.1002/jgt.21757},
  author  = {Cheyu Lin and Tsai-Lien Wong and Xuding Zhu},
}

@article{candrakova2014,
  title = {k-Regular graphs with the circular chromatic index close to k},
  journal = {Discrete Mathematics},
  volume = {322},
  pages = {19-25},
  year = {2014},
  issn = {0012-365X},
  doi = {10.1016/j.disc.2013.12.021},
  author = {Barbora Candráková and Edita Máčajová}
}

@article{Kaiser2004,
  title = {A revival of the girth conjecture},
  journal = {Journal of Combinatorial Theory, Series B},
  volume = {92},
  number = {1},
  pages = {41-53},
  year = {2004},
  doi = {10.1016/j.jctb.2004.04.003},
  author = {Tomáš Kaiser and Daniel Král' and Riste Škrekovski}
}

@article{Leven1983NP,
  title   = {{NP}-completeness of finding the chromatic index of regular graphs},
  author  = {Leven, Daniel and Galil, Zvi},
  journal = {Journal of Algorithms},
  volume  = {4},
  number  = {1},
  pages   = {35--44},
  year    = {1983},
  doi     = {10.1016/0196-6774(83)90032-9},
  publisher={Elsevier}
}

@phdthesis{ghebleh2007theorems,
  title={Theorems and Computations in Circular Colourings of Graphs},
  author={Ghebleh, Mohammad},
  year={2007},
  school={Simon Fraser University}
}

@inproceedings{bernat2024circular,
  title={Circular chromatic index of small snarks},
  author={Bern{á}t, Du{š}an and Maz{á}k, J{á}n},
  booktitle={Proceedings of the 24th Conference Information Technologies – Applications and Theory (ITAT 2024)},
  editor={Ciencialov{á}, Lucie and Hole{č}n{á}, Martin and Jajcay, R{ó}bert and Jajcayov{á}, Tatiana and Ma{č}aj, Martin and Mr{á}z, Franti{š}ek and Ostert{á}g, Richard and Pardubsk{á}, Dana and Pl{á}tek, Martin and Stanek, Martin},
  year={2024},
  volume={3792},
  pages={150-155},
  publisher={CEUR-WS.org},
  url={https://ceur-ws.org/Vol-3792/paper18.pdf}
}

@article{meringer1999fast,
  title={Fast Generation of Regular Graphs and Construction of Cages},
  author={Meringer, M.},
  journal={Journal of Graph Theory},
  volume={30},
  number={2},
  pages={137--146},
  year={1999},
  publisher={Wiley Online Library}
}

@article{Hackmann2004,
  title   = {The circular chromatic index},
  author  = {Hackmann, Andrea and Kemnitz, Arnfried},
  journal = {Discrete Mathematics},
  volume  = {286},
  number  = {1–2},
  pages   = {89-93},
  year    = {2004},
  doi     = {10.1016/j.disc.2003.11.050}
}

@article{ghebleh2007,
  title = {The circular chromatic index of Goldberg snarks},
  journal = {Discrete Mathematics},
  volume = {307},
  number = {24},
  pages = {3220-3225},
  year = {2007},
  issn = {0012-365X},
  doi = {10.1016/j.disc.2007.03.047},
  author = {Mohammad Ghebleh}
}

@article{McKay2014,
  author = {McKay, Brendan D. and Piperno, Adolfo},
  title = {Practical Graph Isomorphism, II},
  journal = {Journal of Symbolic Computation},
  volume = {60},
  pages = {94--112},
  year = {2014},
  doi = {10.1016/j.jsc.2013.09.003}
}

@article{nadolski,
    title = {The circular chromatic index of some class 2 graphs},
    journal = {Discrete Mathematics},
    volume = {307},
    number = {11},
    pages = {1447-1454},
    year = {2007},
    note = {The Fourth Caracow Conference on Graph Theory},
    issn = {0012-365X},
    doi = {https://doi.org/10.1016/j.disc.2005.11.092},
    url = {https://www.sciencedirect.com/science/article/pii/S0012365X06007412},
    author = {Adam Nadolski}
}

@article{genreg,
  title={Fast Generation of Regular Graphs and Construction of Cages},
  author={Meringer, Marcus},
  journal={Journal of Graph Theory},
  volume={30},
  pages={137--146},
  year={1999},
  publisher={Wiley Online Library}
}

@article{vizing1965,
  title={The chromatic class of a multigraph},
  author={Vizing, V. G.},
  journal={Kibernetika (Kiev)},
  volume={1},
  number={3},
  pages={29--39},
  year={1965}
}

@article{Cordella2004,
  author    = {Luigi P. Cordella and Pasquale Foggia and Carlo Sansone and Mario Vento},
  title     = {A (Sub)Graph Isomorphism Algorithm for Matching Large Graphs},
  journal   = {IEEE Transactions on Pattern Analysis and Machine Intelligence},
  volume    = {26},
  number    = {10},
  pages     = {1367--1372},
  year      = {2004},
  doi       = {10.1109/TPAMI.2004.75}
}

@misc{bakalarka,
  author       = {Zrubák, Filip},
  title        = {Circular chromatic index of 4-regular graphs ({B}c. thesis)},
  school       = {Comenius University in Bratislava, Faculty of Mathematics, Physics and Informatics, Department of Computer Science},
  year         = {2025},
  url          = {http://www.dcs.fmph.uniba.sk/bakalarky/registracia/Detail.php?id=518},
  note         = {Supervisor: Ján Mazák}
}

@phdthesis{mazakthesis,
  title={Circular Edge-Colourings of Cubic Graphs},
  author={Mazák, Ján},
  year={2011},
  school={Comenius University in Bratislava, Faculty of Mathematics, Physics and Informatics}
}

@article{mazaklukotka,
author = {Lukoťka, Robert and Mazák, Ján},
title = {Cubic Graphs with Given Circular Chromatic Index},
journal = {SIAM Journal on Discrete Mathematics},
volume = {24},
number = {3},
pages = {1091-1103},
year = {2010},
doi = {10.1137/090752316}
}

@article{hackmann,
  title = {The circular chromatic index},
  journal = {Discrete Mathematics},
  volume = {286},
  number = {1},
  pages = {89-93},
  year = {2004},
  note = {Cycles and Colourings},
  issn = {0012-365X},
  doi = {10.1016/j.disc.2003.11.050},
  author = {Andrea Hackmann and Arnfried Kemnitz}
}

@article{fiolvilaltella,
  title={A simple and fast heuristic algorithm for edge-coloring of graphs},
  author={Fiol, M.A. and Vilaltella, J.},
  journal={Advances and Applications in Discrete Mathematics},
  volume={12},
  number={2},
  pages={193--202},
  year={2013},
  doi={10.17654/ADM012020193}
}

@misc{kissat,
  title={Kissat {SAT} {S}olver},
  author={Biere, Armin and others},
  year={2025},
  howpublished={Available at \url{https://github.com/arminbiere/kissat}}
}

@article{zhusurvey2001,
  title   = {Circular chromatic number: A survey},
  author  = {Xuding Zhu},
  journal = {Discrete Mathematics},
  volume  = {229},
  number  = {1-3},
  pages   = {371-410},
  year    = {2001},
  doi     = {10.1016/S0012-365X(00)00217-X}
}

@incollection{zhusurvey2006,
  title={Recent Developments in Circular Colouring of Graphs},
  author={Zhu, Xuding},
  booktitle={Topics in Discrete Mathematics},
  series={Algorithms and Combinatorics},
  volume={26},
  year={2006},
  pages={497--550},
  publisher={Springer, Berlin, Heidelberg},
  doi={10.1007/3-540-33123-8_21}
}

@article{afshani,
  title={Circular chromatic index of graphs of maximum degree 3},
  author={Afshani, Peyman and Ghandehari, Mahsa and Ghandehari, Mahya and Hatami, Hamed and Tusserkani, Ruzbeh and Zhu, Xuding},
  journal={Journal of Graph Theory},
  volume={49},
  number={4},
  pages={325--335},
  year={2005},
  doi={10.1002/jgt.20086}
}

@inproceedings{sat2024,
  author       = {Armin Biere and Tobias Faller and Katalin Fazekas and Mathias Fleury and Nils Froleyks and Florian Pollitt},
  title	       = {{CaDiCaL}, {Gimsatul}, {IsaSAT} and {Kissat} Entering the {SAT Competition 2024}},
  editor       = {Marijn Heule and Markus Iser and Matti J{\"a}rvisalo and Martin Suda},
  booktitle    = {Proc.~of {SAT Competition} 2024 -- Solver, Benchmark and
                  Proof Checker Descriptions},
  volume       = {B-2024-1},
  series       = {Department of Computer Science Report Series B},
  publisher    = {University of Helsinki},
  year	       = 2024,
  pages	       = {8-10},
}

@article{Shannon1949,
  author    = {Shannon, C. E.},
  title     = {A theorem on coloring the lines of a network},
  journal   = {Journal of Mathematics and Physics},
  volume    = {28},
  number    = {1-4},
  pages     = {148--152},
  year      = {1949},
  publisher = {Massachusetts Institute of Technology (MIT)},
  doi       = {10.1002/sapm1949281148}
}

@article{macaj2013asymptotic,
  title = {Asymptotic Lower Bounds on Circular Chromatic Index of Snarks},
  author = {Ma{\v{c}}aj, Martin and Maz{\'a}k, J{\'a}n},
  journal = {The Electronic Journal of Combinatorics},
  volume = {20},
  number = {2},
  pages = {P2},
  year = {2013},
  doi = {10.37236/2388},
  url = {https://www.combinatorics.org/ojs/index.php/eljc/article/view/v20i2p2}
}

\end{document}